\documentclass[10pt]{amsart}
\allowdisplaybreaks

\usepackage[utf8]{inputenc}
\usepackage[english]{babel}
\usepackage{enumerate}
\usepackage{graphicx} 
\usepackage{tikz}

\usepackage{amssymb, amsmath, 
amstext, amsopn, amsthm, amscd, amsxtra, amsfonts,mathrsfs}

\usepackage{colonequals}
\usepackage[color=blue!20,textsize=footnotesize]{todonotes}
\usepackage{csquotes}
\usepackage{hyperref}%
\hypersetup{
	urlcolor = red,
	colorlinks = true,
	linkcolor = blue,
	citecolor = blue,
	linktocpage = true,
	pdftitle = {Controllability on manifolds with boundary},
	pdfauthor = {Erlend Grong, Alexander Schmeding},
	bookmarksopen = true,
	bookmarksopenlevel = 1,
	unicode = true,
	hypertexnames =false
}

\usepackage{cleveref} 

\usepackage[
backend=biber,
style=alphabetic,
maxnames=6,
sorting=nyt,
giveninits=true,
]{biblatex}
\newcounter{dummy} \numberwithin{dummy}{section}
\newtheorem{theorem}[dummy]{Theorem}
\newtheorem{corollary}[dummy]{Corollary}
\newtheorem{lemma}[dummy]{Lemma}
\newtheorem{definition}[dummy]{Definition}
\newtheorem{proposition}[dummy]{Proposition}
\theoremstyle{remark}
\newtheorem{remark}[dummy]{Remark}
\newtheorem{example}[dummy]{Example}

\newcommand{\N}{\ensuremath{\mathbb{N}}}
\newcommand{\R}{\ensuremath{\mathbb{R}}}

\newcommand{\scrF}{\mathscr{F}}
\newcommand{\scrG}{\mathscr{G}}
\newcommand{\scrO}{\mathscr{O}}
\newcommand{\scrU}{\mathscr{U}}

\DeclareMathOperator{\id}{id}

\newcommand{\ve}{\varepsilon}
\DeclareMathOperator{\supp}{supp}

\DeclareMathOperator{\rank}{rank}
\DeclareMathOperator{\spn}{span}

\DeclareMathOperator{\Diff}{Diff}
\DeclareMathOperator{\Gr}{Gr}
\DeclareMathOperator{\Diffpid}{Diff^{\partial , \id}}
\DeclareMathOperator{\VF}{Vec}
\DeclareMathOperator{\ev}{ev}
\DeclareMathOperator{\Evol}{Evol}
\DeclareMathOperator{\evol}{evol}
\DeclareMathOperator{\Lf}{\mathbf{L}}

\DeclareMathOperator{\Orb}{Orb}
\DeclareMathOperator{\intOP}{int}
\newcommand{\coloneq}{\colonequals}
\newcommand{\Frechet}{Fr\'{e}chet }
\newcommand{\str}{\mathrm{str}}

\numberwithin{equation}{section}

\title{Controllability and diffeomorphism groups on manifolds with boundary}
\author{Erlend Grong and Alexander Schmeding}
\date{}

\numberwithin{equation}{section}

\address{University of Bergen, Department of Mathematics, P.O.~Box 7803, 5020 Bergen, Norway}
\email{erlend.grong@uib.no}

\address{Institutt for matematiske fag, NTNU Trondheim, Alfred Getz’ vei 1 Gløshaugen, 7034 Trondheim, Norway}
\email{alexander.schmeding@ntnu.no}

\subjclass[2020]{58D05 (primary); 93B05, 53C17, 58D15}

\keywords{Diffeomorphism group, manifold with smooth boundary, controllability, infinite-dimensional Lie group, exponential map}

\thanks{The first author is supported by the grant GeoProCo from the Trond Mohn Foundation - Grant TMS2021STG02 (GeoProCo). The second author was supported by the Research Council of Norway through project 302831 ”Computational Dynamics and Stochastics on Manifolds” (CODYSMA)}

\begin{document}

\begin{abstract}
In this article we consider diffeomorphism groups of manifolds with smooth boundary. We show that the diffeomorphism groups of the manifold and its boundary fit into a short exact sequence which admits local sections. In other words, they form an infinite-dimensional fibre  bundle. Manifolds with boundary are of interest in numerical analysis and with a view towards applications in machine learning we establish controllability results for families of vector fields. This generalises older results due to Agrachev and Caponigro in the boundary-less case. Our results show in particular that the diffeomorphism group of a manifold with smooth boundary is generated by the image of the exponential map.
\end{abstract}

\maketitle

\tableofcontents

\section{Introduction}

It is well known that the diffeomorphism group of a finite-dimensional manifold is an infinite-dimensional Lie group, see e.g. \cite{MR583436,glo22diff,schmeding_2022}. These groups and their geometric structure are of interest in shape analysis \cite[Chapter 5]{schmeding_2022}, machine learning \cite{CaGaRaS23}, geometric hydrodynamics \cite{KaMaS23}, and control theory \cite{Agrachev,Control14}. In the present article we investigate the structure of diffeomorphism groups for manifolds with boundary and controllability on these manifolds. For a smooth compact manifold $M$ with smooth boundary $\partial M$ the restriction of diffeomorphisms to the boundary induces a short exact sequence of infinite-dimensional Lie groups
\begin{align}\label{eq:sequence}
\mathbf{1} \rightarrow  \Diff_0^{\partial, \id}(M)  \rightarrow \Diff_0(M) \rightarrow \Diff_0(\partial M) \rightarrow \mathbf{1}
\end{align}
Here, $\Diff_0(M)$  is the (unit component of the) diffeomorphism group of a compact manifold $M$ with smooth boundary $\partial M$, $\Diff_0(\partial M)$ is the unit component of the diffeomorphisms of the boundary, and $\Diff_0^{\partial, \id}(M)$ denotes the unit component of the diffeomorphisms of $M$ whose restriction to the boundary equals the identity. Note that in \eqref{eq:sequence}, and in all what follows, we work with the identity component of the respective diffeomorphism groups and suppress the $0$-subscript in the notation. See Section~\ref{sec:diffgp} for more information on this sequence. 

Our first main result which is established in Section~\ref{sec:MFBoundary} is the following fibre bundle structure for the sequence \eqref{eq:sequence}, which in particular establishes that the restriction map is surjective, whence the sequence is exact.

\begin{theorem}
Let $M$ be a smooth compact manifold with smooth boundary $\partial M$. Then the short exact sequence \eqref{eq:sequence} admits local sections in the sense that there exists a smooth section of the restriction map $\Diff(M) \rightarrow \Diff(\partial M)$ on some identity neighbourhood.
\end{theorem}
To our knowledge smooth local sections for the sequence \eqref{eq:sequence} have not previously been described in the literature.

The second main contribution of the article at hand is the controllability of diffeomorphism groups on manifolds with boundary. Our results generalise \cite{Agrachev} to the case of a manifold with boundary. If $X$ is a vector field on a compact manifold $M$ with boundary, then its flow $e^X$ is a well-defined diffeomorphism of $M$ when $X(x) \in T_x\partial M$ for any $x \in \partial M$. The notation $e^X$ alludes to the fact that the time $1$-flow of a vector field is the image of the vector field under the Lie group exponential of $\Diff (M)$, cf.\ Section \ref{sec:diffgp}. If $\scrF$ is a collection of vector fields, write
$$\Gr(\scrF) = \{ e^{Y_1} \circ \cdots  \circ e^{Y_k} \, : \, Y_1, \dots, Y_k \in \mathscr{F} \}$$
for the corresponding subgroup of generated by the vector fields $\scrF$.
\begin{theorem} \label{th:control} Let $(M,g)$ be a compact Riemannian manifold with boundary $\partial M$. Write its Levi-Civita connection as $\nabla$, and let $\nu$ be the outwards normal vector field.
Let $\scrF$ be a family of vector fields on $M$ tangent to $\partial M$ at the boundary, whose flows generate a group $\scrG = \Gr(\scrF)$ of diffeomorphisms. Define
\begin{equation} \label{hatF} \hat{\scrF}= \spn_{C^\infty(M)} \scrF = \left\{ \sum_{j=1}^m f_j Y_j \, : \, f_j \in C^\infty(M), Y_j \in \scrF, m \in \N \right\}\end{equation}
be the $C^\infty(M)$-module generated by $\scrF$, with corresponding group $\hat{\scrG} = \Gr(\hat{\scrF})$. Assume the following.
\begin{enumerate}[\rm (I)]
\item The group $\scrG$ acts transitively on both $\intOP M$ and any connected component of $\partial M$.
\item For each connected component of the boundary, there is a point $x \in \partial M$ in the component, a neighbourhood $V$ of $x$ and a vector field $Z \in \hat{\scrF}$ such that
$$Z|_{V \cap \partial M} =0 \qquad \text{ and } \qquad \langle \nabla_{\nu} Z, \nu \rangle(x) \neq 0.$$
\end{enumerate}
Then $\hat{\scrG}$ equals $\Diff(M)$.
\end{theorem}
\noindent Let us remark the following about this result.
\begin{enumerate}[$\bullet$]
\item Compared to the original results pertaining controllability on a manifold, cf. \cite{Agrachev}, we need to place the additional control (II) for the boundary term. We provide a counterexample to controllability in the absence of condition (II), see Section \ref{sect:counterex}.
\item At first glance, results of Theorem~\ref{th:control} seem geometric, as they involve the outwards normal vector fields and the Levi-Civita connection. However, as explained in Section~\ref{sec:NotGeometric}, a change in Riemannian will not change controllability properties, nor will a change in connection to any arbitrary affine connection. Also notice that stating $\langle \nabla_{\nu} Z, \nu \rangle(x) \neq 0$ is equivalent to assuming that $\nabla_{v} Z \neq 0 \bmod T_x\partial M$ whenever $v \in T_x M$ and $ v \neq 0 \bmod T_x \partial M$. Furthermore, $\nabla Z$ does not depend on the choice of conne
\end{enumerate}
There is a direct consequence of Theorem \ref{th:control} for the diffeomorphism group $\Diff (M)$. Recall first that the Lie group of diffeomorphisms of a (compact or non-compact) manifold is known to be not locally exponential, i.e. the exponential map is not surjective onto any neighbourhood of the unit, cf.\ \cite[Example 3.42]{schmeding_2022}. Nevertheless, the diffeomorphism group of a manifold without boundary is still generated by the image of the exponential map. This follows for example from a result by Thurston \cite{Thu74} which states that $\Diff_0 (M)$ is simple (see \cite{Ban97} or the alternative approach in \cite{HaRaT13}). The group theoretic argument does not work for the diffeomorphism group of a manifold with boundary which is \emph{not simple} (cf.\ \cite{Ryb98} where the weaker condition of perfectness for diffeomorphism groups of manifolds with boundary is discussed): The subgroup $\Diff_0^{\partial, \id}(M)$ of all diffeomorphisms restricting to the identity on the boundary is a non-trivial normal subgroup in $\Diff_0 (M)$. However, the statement is also a consequence of the controllability results in \cite{Agrachev}
which we generalise to manifolds with boundary. Thus if we take $\scrF$ to be the maximal possible set of vector fields at our given space, we obtain the following result.

\begin{corollary}\label{cor:exp_gen}
    Let $M$ be a compact manifold with smooth boundary $\partial M$. The following connected groups are generated by the image of their Lie algebras under the exponential map.
    \begin{enumerate}[\rm (a)]
    \item The identity component of the diffeomorphism group $\Diff(M)$.
    \item The identity component of the group $\Diffpid(M)$ of diffeomorphims that are the identity on the boundary.
    \end{enumerate}
    Their Lie algebras are respectively $\{ X \in \VF(M) \, : \, X(x) \in T\partial M \text{ for any $x\in M$}\}$ and $\{ X \in \VF(M) \, : \, X(x) =0 \text{ for any $x\in M$}\}$.
\end{corollary}

 We remark that despite the result contained in Corollary \ref{cor:exp_gen} being announced as early as \cite{Luk78}, to the best of our knowledge, no proof for this fact is available in the literature. In particular, \cite{Luk78} only announces the result as a consequence of the same methods as in the boundary less case without providing any details. Hence, the content of Corollary \ref{cor:exp_gen}, while being of interest for example in geometric hydrodynamics (see \cite[Section 8]{KaMaS23}), seems to be new for manifolds with boundary. 

The theory developed in this article provides a framework for numerical mathematics and deep learning algorithms on manifolds. The controllability results from \cite{Agrachev} were of use in machine learning algorithms for certain ordinary differential equations on manifolds, cf.\ \cite{EaZaOaP23}. Our results are a first step to establish a theoretic background for controllability in the study of neural networks on manifolds. In \cite{CaGaRaS23}, these networks have been considered on square domains in $\mathbb{R}^2$ motivated by problems from shape analysis. While Theorem \ref{th:control} is not sufficient to treat the example of a square domain (as it is a manifold with corners, but not with smooth boundary), it is a first step towards these applications.

Indeed, one can pose the question as to whether our results generalise in a sensible way to manifolds with corners \cite{MR583436,MRO92} or manifolds with even more general boundary (cf.\ e.g.\ \cite[Appendix A]{AaGaS20}).
It is not hard to see that the controllability results developed here admit a natural generalisation (at least for 2D-polytopes). The issue, both for controllability and sections for the diffeomorphism groups is more delicate for general manifolds with corners. We exclude these considerations from the present work which on purpose restricts to manifolds with smooth boundary. However, similar results as presented can be achieved for certain manifolds with corners. Details of this ongoing joint work with H.\ Gl\"{o}ckner will be provided elsewhere, \cite{GaGaS24}. 

Note also that our results yield information for sub-Riemannian geometry on diffeomorphism groups on manifolds with boundary, cf.\ e.g.\ \cite{AaT17,Arg20} where the case of manifolds without boundary is studied.

The structure of the paper is as follows. In order to build a solid foundation, we need to formulate the Lie group structure for diffeomorphisms of manifolds with a smooth boundary. One would think that such a structure would follow from Michor's result in \cite{MR583436}. However, as \cite{glo22diff} shows, the well known construction is incorrect, even though the statement that the diffeomorphisms form a Lie group remains true. The key problem is its construction, which does not generalise to manifolds with non-empty boundary, see \cite{glo22diff} for counter-examples. As pointed out in \cite{Mic20} the problem could be remedied by using an exponential of a suitable spray (see \cite[Section 5.9]{Mic20}). However, to our knowledge there is currently no reference in the literature which features a correct proof in full generality. As a consequence, we have to work around this problem. Exploiting that each manifold with smooth boundary embedds into its double we give an independent construction of the manifold structure on the diffeomorphism groups, see Section~\ref{sec:MFBoundary}. Then we establish the splitting of the sequence of diffeomorphism groups. In Section \ref{sect:controllability}, we then extend the controllability results to manifolds with boundary. Finally, Section~\ref{sec:Examples} contains several examples showcasing our results or the conditions needed.
\smallskip

\textbf{Acknowledgements} The authors thank H.\ Gl\"{o}ckner for helpful discussions and the idea leading to Proposition \ref{prop:ex:smooth-sect}.  
We would also like to thank K.--H.\ Neeb and P.W.\ Michor for helpful comments concerning this work. Finally, we thank the referees whose insightful comments led to several improvements of the article.

\section{Sequences of diffeomorphism groups on manifolds with boundary} \label{sec:MFBoundary}
For this section, $M$ will denote a smooth compact manifold with smooth boundary $\partial M$. For basic information on these objects and mappings between them, we refer to \cite{Lee13}.

\subsection{Preliminaries on manifolds with smooth boundary}
\label{subsect:prelimSMB}
For later use, we recall several facts on manifolds with smooth boundary: 
\begin{enumerate}[$\bullet$]
\item From \cite[Theorem 9.25]{Lee13}, it follows that $\partial M$ admits a \emph{collar neighbourhood}. By this collar we mean a smooth embedding $h \colon \partial M \times [0,1) \rightarrow M$ such that $h|_{\partial M \times \{0\}}^{\partial M} \colon \partial M \times \{0\} \rightarrow \partial M$ is a diffeomorphism.
\item Compact manifolds with or without boundary are paracompact. Thus every open cover of $M$ has a smooth partition of unity subordinate to itself \cite[Theorem 2.23]{Lee13}.
\item $M$ embeds as a compact submanifold with boundary into \emph{its double} $DM$, the latter being a compact manifold without boundary \cite[Example 9.32]{Lee13}.
\item Finally, using \cite[Theorem 9.34 and Corollary 9.17]{Lee13}), any vector field $X$ that is tangent to the boundary at boundary points can be integrated to a global flow $e^{tX}$ of diffeomorphisms of $M$. 
\end{enumerate}
In what follows, we write $\VF(TM)$ for the sections of the tangent bundle, i.e., vector fields in general on $M$, and $\VF(M;\partial M)$ for the vector fields on $M$ that satisfy $X(x) \in T_x \partial M$ for every $x \in \partial M$.

\subsection{Spaces and manifolds of smooth mappings} \label{mfd_withoutbdr}
For smooth manifolds $M,N$ we denote by $C^\infty (M,N)$ the set of smooth functions from $M$ to $N$. If $M$ is compact, we endow $C^\infty (M,N)$ with the compact open $C^\infty$-topology, see e.g. \cite{MR583436,schmeding_2022}. 

In the following proofs we need certain subsets of mappings to be open in the compact open $C^\infty$-topology. For this, let $M$ be a compact manifold and $N$ be an arbitrary manifold (both possibly with boundary if nothing else is said). Then the following sets are open in the compact open $C^\infty$-topology (cf.\ \cite[Chapter 2]{schmeding_2022} for the case of empty boundary and \cite[Section 2.1]{Hir76} for the general case):
\begin{enumerate}[$\bullet$]
\item The set of submersions $\text{Sub} (M,N)$ and the set of embeddings $\text{Emb}(M,N)$ in $C^\infty (M,N)$. 
\item If $M$ has no boundary, the diffeomorphism group $\Diff (M) \subseteq C^\infty (M,M)$ is open.
\item If $M$ has smooth boundary, the diffeomorphism group $\Diff(M)$ is open in $\{f \in C^\infty (M,M)\mid f(\partial M) \subseteq \partial M\}$
\end{enumerate}
This structure also turns the diffeomorphism group into a topological group. 
We shall from now on always assume that the source manifold $M$ is compact, possibly with smooth boundary. Further, we assume that $N$ is endowed with a local addition~$\Sigma$.

\begin{definition}
A \emph{local addition} is a smooth map $\Sigma\colon TN \supseteq O \to N$, where
\begin{enumerate}[\rm (i)]
\item $O$ is an open neighbourhood containing the zero section;
\item $\Sigma(0_x) = x$ for any $x \in M$;
\item If $\pi:TM \to M$ and $\tilde \Sigma = (\pi, \Sigma)$, then $\tilde \Sigma$ is a diffeomorphism from $O$ to its image in $M \times M$
\end{enumerate}
We remark that any finite dimensional manifold has a local addition given by the exponential map of any choice of Riemannian metric.
\end{definition}

For a manifold $M$ (possibly with boundary) we write
$$\VF(M)\coloneq C_{\id}^\infty(M,TM) =\{ X \in C^\infty (M,TM) \mid \pi \circ X = \id_M\}.$$
Topologise $\VF(M)$ as a subspace of $C^\infty(M,TM)$, by viewing $TM$ as a subset of the tangent bundle of the double $DM$ of $M$. We may argue, as in \cite[A.1 and A.3 (c)]{AaGaS20}, that the subspace topology of $\VF(M) \subseteq C^\infty (M,TM)$ turns the vector fields into a locally convex space. In the following we will then consider the closed subspace $\VF(M;\partial M) \subseteq \VF(M)$ of all vector fields tangent to the boundary.

\begin{remark}\label{rem:expolaw_sect}
Let $L$ be a smooth manifold possibly with boundary. Embedding $M$ in its double $D(M)$, we identify via the inclusion $TM = TD(M)|_M$, whence from \cite[Lemma A.2]{AaGaS20} we deduce that a mapping $g \colon L \rightarrow \VF(M)$ is smooth if and only if the associated mapping $g^\wedge \colon L \times M \rightarrow TM, g^\wedge (\ell,m)\coloneq g(\ell)(m)$ is smooth. 
Combining this observation with the inclusion $TM \rightarrow TD(M)$ we can identify $\VF(M;\partial M)$ with the closed subset $$\{F \in C^\infty (M,TD(M))\mid \forall x \in M , F(x)\in T_xM \text{ and } F(y)\in  T_y\partial M,\ \forall y \in \partial M\} $$
\end{remark}

\textbf{The manifold structure for $C^\infty (M,N)$ for $N$ without boundary} Assume that $N$ is a smooth manifold without boundary with local addition $\Sigma$, with $\tilde \Sigma = (\pi, \Sigma)$. We can then define \emph{the canonical charts} for $C^\infty(M,N)$, see \cite[Chapter~A.9]{AaGaS20}. By considering smooth mappings $$g \in C^\infty_f (M,TN) \coloneq \{ F \in C^\infty (M,TN) \mid \pi_N \circ F = f\}$$ such that $g(M)\subseteq O$, and $h \in C^\infty (M,N)$ such that $(f,h)(M)\subseteq \tilde \Sigma(O)$, we can define the mutually inverse mappings
\begin{equation}\label{can:charts}
\varphi_f (g) \coloneq \Sigma \circ g \text{ and } \varphi_f^{-1}(h)\coloneq {\tilde \Sigma}^{-1}(f,h), \qquad \text{ for all }f \in C^\infty (M,N).
\end{equation}
The map $\varphi_f$ is defined on an open subset of $C^\infty_f (M,TN)$, which is a closed vector subspace of $C^\infty (M,TN)$ whence a locally convex space. Similarly, the image of $\varphi_f$ is open in the compact-open $C^\infty$-topology. 
As \cite[Appendix A]{AaGaS20} shows, this atlas turns $C^\infty(M,N)$ into an infinite-dimensional manifold. 

\begin{remark}[$C^\infty (M,N)$ is a canonical manifold] If $N$ does not have a boundary \cite[Lemma A.10]{AaGaS20} shows that the manifold  of mappings is canonical in the sense of \cite[Section 2.3]{schmeding_2022}:

Any $\gamma \colon L \rightarrow C^\infty (M,N)$ is smooth if and only if $\gamma^\wedge \colon L \times M \rightarrow N, \gamma^\wedge (\ell,m)=\gamma(\ell)(m)$ is smooth. Denoting by $\partial_L\gamma^\wedge$ the partial derivative of a map with respect to the $L$-component, there is a natural identification $TC^\infty (M,N) \cong C^\infty (M,TN)$ via $(T\gamma)^\wedge = \partial_L \gamma^\wedge$.
\end{remark}

\subsection{Diffeomorphism groups of manifolds with boundary}\label{sec:diffgp}

In this section we construct the Lie group structure for the diffeomorphism group of the manifold $M$. The point will be that we follow an idea from \cite{EM70} to construct the manifold structure of $\Diff (M)$ in the case $M$ has smooth boundary. However, recall first that if $M$ has no boundary, then one can exploit that $\Diff (M) \subseteq C^\infty (M,M)$ is an open submanifold of the manifold constructed in the last section. It is well known that $\Diff(M)$ with this structure is a Lie group, \cite[Example 3.5]{schmeding_2022}. Unfortunately, we do not have a manifold structure available on $C^\infty (M,M)$ if $M$ has boundary, so we need to work around this. 

\textbf{Assume now that $M$ has non-empty smooth boundary $\partial M$.} We embedd $M$ in its double $D(M)$ and obtain a smooth manifold structure on $C^\infty (M,D(M))$ as in the previous section. Recall that elements in $\Diff (M)$ restrict to diffeomorphisms of the boundary $\partial M$. Further, by composing elements in $\Diff (M)$ with the inclusion $\iota \colon M \rightarrow D(M)$, we can identify $\Diff (M)$ as a subset of $C^\infty (M,D(M))$ and indeed as a subset of the open set $\text{Emb}(M,D(M))$. We will suppress the inclusion in the notation and simply write $\eta \in \Diff (M) \subseteq \text{Emb}(M,D(M)) \subseteq C^\infty (M,D(M))$.

The manifold structure on the smooth functions does not depend on the choice of the local addition, cf. \cite[Lemmata C.11 and 2.16 (b)]{schmeding_2022}. We may thus assume that the local addition is given by the Riemannian exponential map $\exp$ of a metric for which $\partial M$ is a totally geodesic submanifold of $D(M)$, \cite[Lemma 6.4]{EM70}. Denote by $\varphi_\eta$ the canonical parametrisation of $C^\infty (M,D(M))$ constructed in \eqref{can:charts} with respect to the Riemannian exponential map $\exp$. The idea is to prove that we can find an open neighbourhood on which the canonical parametrisations restrict to submanifold charts.

\textbf{Construction of an open section neighbourhood}
Since $\eta \in \text{Emb}(M,D(M))$ and this set is open, there exists an open $0$-neighbourhood $W \subseteq C^\infty_\eta (M,TD(M))$ which gets mapped by $\varphi_\eta$ into $\text{Emb}(M,D(M))$. Note that this in particular implies that for every $X \in W$ we have that $\varphi_{\eta}(X)|_{\partial M} \colon \partial M \rightarrow D(M)$ is an embedding (whence also a submersion for dimensional reasons).
Further, we observe that for every $X \in W$ 
and all $x \in \partial M$ the map $T_x \varphi_\eta(X) \colon T_x \partial M \rightarrow T_{\eta (x)}D(M)$ has rank $\text{dim } \partial M$.
Note that if we knew that the restriction of $\varphi_\eta (X)$ to $\partial M$ takes its image in $\partial M$, then the restriction becomes a submersion. 

Pick for every connected component $C_{0,i}$ of $M\setminus \partial M$ a point $x_{0,i} \in C_{0,i}$ and an open neighbourhood $\eta (x_{0,i}) \in  O_{0,i}\subseteq \eta (C_{0,i}) \subseteq M\setminus \partial M$. If the boundary $\partial M$ consists of several components $C_{1,i}$ we pick for every $i$ a point $x_{1,i} \in C_{1,i}$ and an open neighbourhood $O_{1,i}$ of $\eta (x_{1,i})$, such that $O_i \cap \partial M \subseteq C_{1,i}$, where $C_j = \eta(C_i)$. Shrinking $W$ further, we may assume that $\varphi_\eta (X)(x_{k,i})\in O_{k,i}$ for every $k \in \{0,1\}$, $i$ and $X \in W$.

\textbf{The submanifold chart around $\eta \in \Diff (M)$.}
Define the set  
$$S_\eta \coloneq W \cap \{F \in C^\infty_\eta (M,D(M)) | \forall x \in \partial M,\quad F(x)\in T(\partial M)  \}$$
and recall that we just intersected $W$ with a closed vector subspace of the sections. We claim, similar to \cite[Lemma 6.6]{EM70}, that $\varphi_\eta$ takes $S_\eta$ to $\Diff (M)$,
giving us a smooth (sub)manifold atlas by the restriction of the canonical charts and inducing a manifold structure on $\Diff (M)$.

So let us prove the claim: By construction if $X \in S_\eta$, then $X(x) \in T_{\eta(x)}\partial M$ for every boundary point $x$. We obtain $\varphi_\eta (X)(x) \in \partial M$ as the boundary is totally geodesic and $\varphi_\eta$ is by \eqref{can:charts} given as postcomposition with $\exp$. This implies by choice of the neighbourhood $W$ that for every $X \in S_\eta$, the map $\varphi_\eta(X)|_{\partial M}^{\partial M}$ makes sense and
\begin{enumerate}[$\bullet$]
\item is a smooth embedding and submersion by choice of $W$;
\item maps every connected component $C_i$ into the connected component $\eta(C_i)$.
\end{enumerate}
We deduce that $\varphi_\eta (X)$ induces a diffeomorphism of $\partial M$ if we (co)restrict it to the boundary (cf.\ the argument in \cite[Corollary 2.8]{schmeding_2022}).
So far, we obtained embeddings $\varphi_\eta (X) \colon M \rightarrow D(M)$ which induce diffeomorphisms on $\partial M$. By construction, $D(M) \setminus \partial M = M\setminus \partial M \sqcup D$, where $D$ in the disjoint union is a copy of $M\setminus \partial M$. Again, by construction of the open set $W$ we already know that for every $X \in W$, the map $\varphi_\eta (X)$ takes a point from every component of $M\setminus \partial M$ to a component of $M\setminus \partial M$. We already know that $\varphi_\eta (X)$ induces a diffeomorphism of the boundary $\partial M$, whence by continuity we must have $\varphi_\eta (X)(M) \subseteq M$. As $\varphi_\eta (X)$ is also an embedding, we deduce that it is an element in $\Diff (M)$ thus proving the claim. \smallskip

\textbf{Lie group structure}
Since $M$ is compact, the manifold $\Diff (M)$ is modelled on \Frechet spaces. To establish the smoothness of the group operations it therefore suffices to check that they are smooth in the convenient sense (see \cite[Appendix A.7]{schmeding_2022}). Due to the results in the previous section, the argument from \cite[Proof of Theorem 43.1]{KaM97} is applicable with obvious minor modifications. Thus $\Diff (M)$ is an infinite-dimensional Lie group with associated Lie algebra $\Lf (\Diff(M))=\VF(M;\partial M)$. Here $\VF(M;\partial M)$ is the Lie algebra of vector fields on $M$ tangent to the boundary with the negative of the usual Lie bracket of vector fields.

Let us now recall the notion of a regular Lie group $G$ with Lie algebra $\Lf (G)$: We say $G$ is regular if  for every $\eta \colon [0,1] \rightarrow \Lf (G)$ smooth, the ordinary differential equations
\begin{equation}\label{Lietypeeq}
\begin{cases}
\dot{\gamma}(t)&=\eta(t) \cdot \gamma(t) =: \delta^r(\gamma) \\
\gamma(0)&=\id_{M}
\end{cases}\end{equation}
admits a unique solution $\Evol (\eta) \colon [0,1] \rightarrow G$ and the solution depends smoothly on $\eta$. 

\begin{lemma}\label{lem:diffgp}
Let $M$ be a compact manifold with smooth boundary. Then the Lie group $\Diff(M)$ has the following properties.
\begin{enumerate}[\rm (i)]
\item A map $g \colon L \rightarrow \Diff (M)$ is smooth if and only if $g^\wedge \colon L \times M \rightarrow M$ is smooth.
\item  $\Diff (M)$ is a regular Lie group and using the identification $\Lf(\Diff(M))=\VF(M;\partial M)$, $\Evol$ sends a time-dependent vector field $X$ to its flow $\mathrm{Fl}^X=e^X$.
\end{enumerate}
\end{lemma}
\begin{proof}
Since $\Diff(M)$ is a closed submanifold of $C^\infty (M,D(M))$, $g \colon L \rightarrow \Diff(M)$ is smooth if and only if the mapping is smooth as a map to $C^\infty (M,D(M))$. Thus (i) follows from $C^\infty (M,D(M))$ being a canonical manifold.

Now for regularity of the Lie group, we can mimic the classical proofs for diffeomorphisms of manifolds without boundary found in \cite[Example 3.36]{schmeding_2022}. By propterty (i), we can identify the ordinary differential equation \eqref{Lietypeeq} on $\Diff (M)$ with a time dependent differential equation on a product of finite-dimensional compact manifolds. Existence of solutions and smooth parameter dependence follows then by the usual arguments, see e.g. \cite[chapter 9]{Lee13}.
\end{proof}

\subsection{Diffeomorphism groups as a principle bundle} As outlined in the introduction are we interested in the sequence of groups \eqref{eq:sequence}, i.e.,
$$\mathbf{1}\rightarrow \Diff_0^{\partial, \id} (M) \xrightarrow{\iota} \Diff_0 (M) \xrightarrow{R}\Diff_0 (\partial M) \rightarrow \mathbf{1}$$
where the mapping $\iota$ is inclusion and $R$ is given by restricting a diffeomorphism of~$M$ to the boundary. Let us first justify why this sequence indeed is an exact sequence of Lie groups: Clearly the first three terms in this sequence are exact on the level of group morphisms. So far we have neither established that the left hand term is a Lie subgroup, nor that $R$ is smooth. We shall do so in a moment. Further, it is not obvious that $R$ is surjective. We shall construct preimages to arbitrary elements in the unit component of the diffeomorphism group of the boundary in Lemma \ref{la:diff_ext}.

Let us first establish that $R$ is smooth. Recall from Section~\ref{sec:diffgp} that $\Diff (M)$ is a closed submanifold of $C^\infty (M,D(M))$. Furthermore, \cite[Proposition 2.10]{AaGaS20} establish that $\Diff (\partial M)$ is diffeomorphic to a closed submanifold of $C^\infty (\partial M,D(M))$. This identifies $R$ as the (co-)restriction of the smooth map
$$C^\infty (M,D(M)) \rightarrow C^\infty (\partial M, DM),\quad  \phi \mapsto \phi \circ \text{inc}_{\partial M},$$
since domain and codomain are canonical manifolds of mappings, smoothness follows from \cite[Exercise 2.3.1 (a)]{schmeding_2022}.
Hence $R$ is a Lie group morphism. We will now explicitly construct submanifold charts for its kernel and prove that the kernel is a closed Lie subgroup. 

\begin{lemma}\label{la:subgroup}
The inclusion $\Diffpid (M) \rightarrow \Diff (M)$ turns $\Diffpid(M)$ into a closed Lie subgroup of $\Diff(M)$. Moreover, $\Diffpid (M)$ is a regular Lie group.
\end{lemma}

\begin{proof}
To see that $\Diffpid(M)$ becomes a submanifold of $\Diff (M)$, recall that the canonical charts of $\Diff (M)$ arise as restrictions of \eqref{can:charts} to suitable subspaces where the local addition is the Riemannian exponential on $D(M)$ and $\partial M$ is a totally geodesic submanifold of $D(M)$. In particular, the Riemannian exponential map vectors tangent to the boundary to the boundary. Now similar to \cite[Remark 5.6]{glo22diff} the canonical charts restrict to submanifold charts of $\Diffpid (M)$. For the chart around the identity of $\Diff (M)$ this yield the following:
$$\varphi_{\id}^{-1} (g)  \in \{X \in \VF (M) \mid X(b)=0, \forall b \in \partial M\} \text{ if and only if } g \in \Diffpid (M).$$ 
We deduce that $\Diffpid(M)$ is a subgroup and a closed submanifold of $\Diff (M)$, hence a closed Lie subgroup. 

For the regularity, note that $\Diffpid (M)$ is a closed Lie subgroup of the regular Lie group $\Diff (M)$. Hence \cite[Lemma B.5]{SaW16} implies that $\Diffpid(M)$ will be regular if we can shows that solutions to the Lie type equations \eqref{Lietypeeq} exist in $\Diffpid (M)$ for every smooth curve $\eta \colon [0,1] \rightarrow \Lf(\Diffpid(M))$. Using the inclusion $\Lf(\Diffpid(M))\rightarrow \Lf(\Diff(M))=\VF(M;\partial M)$, we only need to show that the solution $\Evol (\eta)$ of \eqref{Lietypeeq} on $\Diff(M)$ stays in the subgroup $\Diffpid (M)$ for curves taking values in the Lie subalgebra $\{X \in \VF (M) \mid X(x)=0, \text{for all $x \in \partial M$}\}$ associated by the inclusion to the closed subgroup $\Diffpid(M)$. However, by \Cref{lem:diffgp}, $\Evol(\eta)$ is the flow of the time-dependent vector field on $M$ associated to $\eta$. Since $\eta$ vanishes on every point of the boundary $\partial M$, the flow associated to $\eta$ is the identity on the boundary, whence an element in $\Diffpid(M)$. This completes the proof. 
\end{proof}

Recall the following standard definitions, see e.g. \cite{Michor08,bauer2023regularity}:
\begin{definition}
Consider the short exact sequence
$$\mathbf{1} \rightarrow  G \xrightarrow{\iota} H \xrightarrow{\pi} K \rightarrow \mathbf{1}$$
of Lie group morphisms. The sequence
\begin{enumerate}[\rm (i)]
\item \emph{admits local sections} if $\pi$ admits a local smooth section $s$ near the identity (whence at every point), and $\iota$ is initial, i.e.\ any $g \colon M \rightarrow H$ with image in $\iota(G)$ is smooth if and only if $g$ is smooth as a map to $\iota (G)$.
 \item \emph{admits local retractions} if there exists a smooth map $r \colon U \rightarrow G$, where $U$ is an open identity neighbourhood such that $r(\iota (g))=g$ for all $g\in \iota^{-1}(U)$ and $\pi$ is final, i.e.\ any $f \colon K \rightarrow L$ is smooth if and only if $f\circ \pi$ is smooth.
\item is \emph{split} if there exists a smooth group homomorphism $S \colon K\rightarrow H$ such that $\pi\circ S = \id_K$.
\end{enumerate}
Note that a short exact sequence of Lie groups admits local sections if and only if it admits local retractions since $h = s(\pi(h))\cdot \iota (r(h))$. Further, if the sequence admits local sections, then  $H$ is locally diffeomorphic to $G \times K$ via $(r,\pi)$ with local inverse $(\iota \circ \mathrm{pr}_1)\cdot (s \circ \mathrm{pr}_2)$. In particular, if the sequence splits, then $H$ is diffeomorphic to a semidirect product of $G$ and $K$ (see \cite[15.12]{Michor08}).
\end{definition}

We shall now prove that the sequence \eqref{eq:sequence} admits local sections, i.e.\ that there is a smooth section of the projection $R$ on some identity neighbourhood of $\Diff (\partial M)$. For this we need some preparation.

\subsection{Construction of a boundary thickening}\label{subsect_collar_thick}
Let $h \colon \partial M \times [0,1) \rightarrow M$ be a smooth collar for the boundary as in Section~\ref{subsect:prelimSMB}. Then for every $0<r< 1$, the set $\partial M \times [0,r]$ is a submanifold with smooth boundary of $\partial M \times [0,1)$. Since $M$ is compact, so is the closed subset $\partial M$ and we deduce that for each $r$, the set $N_r \coloneq h(\partial M \times [0,r])$ is a compact submanifold with boundary of $M$. We can think of $N_r$ as a thickening of the boundary. Now pick and fix once and for all numbers $s= \tfrac{1}{2}<s'<r$ and a smooth cut off function $C \colon [0,r] \rightarrow [0,1]$ such that $C|_{[0,s]}\equiv 1$ and $C|_{[s',r]} \equiv 0$. This yields a smooth bump function 
\begin{equation} \label{ChiFunction}
\chi \colon M \rightarrow [0,1], \quad k \mapsto \begin{cases} 0 & \text{ if } k \not \in N_h \\
C (\mathrm{pr}_2 \circ h^{-1}(k)) & \text{ if } k \in N_h
\end{cases}\end{equation}
which by construction is constant equal to $1$ on the neighbourhood $N_s$ of $\partial M$ and vanishes outside of $N_{s'}$. For a vector field $X \in \VF(\partial M)$ we can now define an extension of the vector field to $M$ via the following formula
\begin{equation}\label{ext:VF}
\hat{X} (k) \coloneq \chi(k) \cdot \left(Th (X,\mathbf{0})(h^{-1}(k))\right),
\end{equation}
where $\mathbf{0}$ denotes the zero vector field in $\VF([0,r])$.
Note the abuse of notation in \eqref{ext:VF}: Since $\chi(k)=0$ for all $k \not \in N_r$, we interpret the value of $\hat{X}$ at such a $k$ as $0$ even though $h^{-1}$ is not defined for these values. 

Using the auxiliary functions from \eqref{ChiFunction}, we can now construct a continuous linear extension map for vector fields on the boundary.

\begin{lemma}\label{lem:ext_operator}
The extension $\hat{X}$ of vector fields $X$ on the boundary $\partial M$ constructed in \eqref{ext:VF}, yields an associated map 
$$\theta \colon \VF (\partial M ) \rightarrow \VF (M;\partial M), \quad X \mapsto \hat{X}$$
which is continuous and linear.
\end{lemma}

\begin{proof}
Clearly $\theta$ is linear and takes its image in $\VF(M;\partial M)$ (cf. \eqref{ext:VF}). So all we have to prove is continuity of $\theta$ as a mapping to $\VF(M)$. For this, we shall consider  the spaces of vector fields $\VF(\partial M)$, $\VF(\partial M \times [0,r])$, $\VF(N_r)$ and $\VF(M)$. All of these spaces are \Frechet spaces with respect to 
the compact open $C^\infty$-topology, \cite[Remark A.3 (c)]{AaGaS20}. In particular, $\VF(\partial M \times [0,r]) \cong \VF(N_r)$ via the isomorphism of locally convex spaces $I (Y)\coloneq Th \circ Y \circ h$. Consider now 
\begin{align}\label{ext:trivial}
E \colon \VF(\partial M) \rightarrow \VF(\partial M \times [0,r]),\quad X \mapsto (X,\mathbf{0}).
\end{align}
Since $[0,r]$ is compact, \eqref{ext:trivial} is continuous, cf. e.g.\ \cite[Corollary 11.10 with Remark 11.5]{MR583436}. With the identification $I$ of $\VF(N_r)$ and $\VF(\partial M \times [0,r])$ we rewrite $\theta$ now as
$\theta(X)= m_\chi \circ I \circ E(X)$ where $I,E$ are already known to be continuous linear and $m_\chi$ is the map defined via
$$m_\chi \colon \VF(N_r) \rightarrow \VF(M), \quad Y \mapsto \chi \cdot Y$$
and the vector field on the right hand side is trivially extended by the zero vector field beyond $N_r$. To see that this mapping is continuous, denote by $dN_r$ the topological boundary in $M$, i.e, $dN_r = h(\partial M \times \{r\})$. Recall from \cite[Lemma A.4]{AaGaS20} that point evaluations are continuous on the space of vector fields. As derivations are continuous with respect to the compact open $C^\infty$-topology (cf.\ \cite[p. 313, 1.12]{AaGaS20}), we obtain a closed vector subspace 
$$\VF_{d\text{flat}}(N_r)= \{Y \in \VF(N_r) \mid \forall b \in d N_r, k \in \N_0, \ T^kY(b)=0\}.$$ 
Now consider the multiplication map,
$$\tilde m_\chi \colon \VF(N_r) \rightarrow \VF(N_r),\quad Y \mapsto \chi \cdot Y.$$
By the exponential law from Remark \ref{rem:expolaw_sect}, $\tilde m_\chi$ is smooth if and only if the associated map $\tilde m_\chi^\wedge \colon \VF (N_r) \times N_r \rightarrow TN_r$, $(X,n) \mapsto \chi (n)\cdot X(n) = \ev ((\text{mult}_\chi)_*(X),n)$ is smooth. Here $(\text{mult}_\chi)_*$ is the pushforward with the smooth map implementing fibrewise multiplication by $\chi$, i.e. $TM \rightarrow TM, v \mapsto \chi(\pi_M(v))\cdot v$. The pushforward is smooth by \cite[Corollary 1.22 and Remark A.3(c)]{AaGaS20}. As evaluation $\ev$ is smooth by \cite[Lemma A.4]{AaGaS20}, we conclude that $\tilde m_\chi$ is smooth and takes its values in the closed vector subspace $\VF_{d \text{flat}}(N_r)$. To establish that $m_\chi$ is continuous it thus suffices to prove that the mapping $\VF_{d \text{flat}}(N_h) \rightarrow \VF(M;\partial M)$ which sends a vector field to its trivial extension by $0$ is continuous. It suffices to test continuity on the basic neighbourhoods of the compact open $C^\infty$-topology. Such a basic open set $\lfloor L, W,\ell\rfloor$ tests whether a function and up to finitely many derivatives $\ell \in \N_0$ take their values on a given compact set $L$ in an open set $W \subseteq T^\ell M$. See \cite[Section 2.1]{schmeding_2022} for details. Now as $N_r$ is closed, $\lfloor L, W,\ell\rfloor$ gives rise to the basic open neighbourhood $\lfloor L\cap N_r, W \cap T^\ell N_r, \ell\rfloor$. It is easy to see that if $Z \in \lfloor L, W,\ell\rfloor$ is an extension of a vector field in $\VF_{d \text{flat}}(N_r)$, then $Z \in \lfloor L\cap N_r, W\cap T^\ell N_r, \ell\rfloor$ and every other vector field in $\lfloor L\cap N_r, W \cap T^\ell N_r, \ell\rfloor$ will also be contained in $\lfloor L, W,\ell\rfloor$. We conclude that $M_\chi$ and thus also $\theta$ is continuous.
\end{proof}

We can use the extension of vector fields to construct an extension of a diffeomorphism as follows. Consider the right logarithmic derivative 
$$\delta^r \colon C^\infty ([0,1],\Diff (\partial M)) \rightarrow C^\infty ([0,1],\VF(\partial M)), c \mapsto (t \mapsto T\rho_{c(t)^{-1}}(c^\prime(t))),$$
where $\rho_g \colon \Diff (\partial M) \rightarrow \Diff (\partial M)$ denotes right multiplication with an element $g$. Recall that for curves starting at the identity, the evolution $\Evol$ sending a (time-dependent) vector field to its flow inverts $\delta^r$. We now recall that every $\psi$ in the identity component of $\Diff (\partial M)$ admits a smooth path $c_\psi \colon [0,1] \rightarrow \Diff (\partial M)$ such that $c_\psi(0)=\id_{\partial M}$ and $c_\psi (1) = \psi$. Then we obtain a diffeomorphism extending $\psi$ via the following construction:

\begin{lemma}\label{la:diff_ext}
Let $\psi \in \Diff (\partial M)$ be a diffeomorphism in the identity component of $\Diff (\partial M)$ with associated smooth path $c_\psi$. Taking the evolution $\Evol$ in $\Diff (M)$, the diffeomorphism 
$$\hat{\psi} \coloneq \Evol (\theta \circ \delta^r (c_\psi))(1) \in \Diff (M)$$
satisfies $\hat{\psi}|_{\partial M} = \psi$
\end{lemma}
\begin{proof}
For $k \in \partial M$ and $t \in [0,1]$, we have by construction of the extension operator $\theta$ the following relation
$$\theta (\delta^r (c_\psi)(t))(k)=\delta^r (c_\psi)(t)(k).$$
The evolution in $\Diff (M)$ is (as the evolution in $\Diff (\partial M)$) the map sending a time dependent vector field to its flow (see Section~\ref{sec:diffgp} and \cite[Example 3.36]{schmeding_2022}). Since the restriction of a flow to a vector field tangential to the boundary is the flow of the vector field restricted to the boundary, and the evolution inverts the logarithmic derivative, we deduce
$$\Evol (\theta (\delta^r (c_\psi))(1)|_{\partial M}^{\partial M} = c_\psi (1) = \psi$$
by construction of $c_\psi$.
\end{proof}
Note that Lemma \ref{la:diff_ext} implies that the map $R$ in \eqref{eq:sequence} is surjective. 

\subsection{Construction of a local section} \label{setup:localisation} Pick now an open identity neighbourhood $\id_{\partial M} \in \Omega \subseteq \Diff (\partial M)$ together with a manifold chart $\kappa \colon \Omega \rightarrow \Omega' \subseteq \VF(\partial M)$ such that $\kappa (\id_M) = 0_{\partial M}$. Shrinking $\Omega$ and $\Omega'$ if necessary, we may assume that $\Omega'$ is an open disc $0$-neighbourhood. If $\psi \in \Omega'$, it thus makes sense to define the smooth curve 
$f_\psi (t) \coloneq t \psi $ defined for $t$ small enough with values in $\Omega'$.
We will now identify $\Omega$ and $\Omega'$ via the diffeomorphism $\kappa$ and suppress this in the notation.
 For $\psi \in \Omega$ and $t$ small we can define the map 
\begin{equation}\label{curve:bdr:primitive}
G \colon \Omega \times  [0,1] \rightarrow \Lf (\Diff(\partial M))=\VF (\partial M),\quad (\psi, t) \mapsto \delta^r (f_\psi)(t).
\end{equation}
Apply the exponential law \cite[Lemma A.4]{AaGaS20} to see that $\Omega \rightarrow C^\infty ([0,1],\Omega), \psi \mapsto f_\psi$ is smooth, whence $G$ is smooth. Applying the exponential law a second time to $\theta \circ G$ (with $\theta$ as in \Cref{lem:ext_operator}, the map $(\theta \circ G)^\vee \colon \Omega \rightarrow C^\infty ([0,1],\VF(M;\partial M))$, $\psi \mapsto \theta(G(\psi)(\cdot))$ is smooth. 

\begin{proposition}\label{prop:ex:smooth-sect}
Using the evolution $\Evol$ of $\Diff (M)$, the map
\begin{equation}\label{splitting:section}
k \colon \Omega \rightarrow \Diff (M) , \qquad \psi \mapsto \Evol ((\theta \circ G)^\vee (\psi))(1).
\end{equation}
is a smooth local section of the projection $R\colon \Diff (M) \rightarrow \Diff (\partial M)$, $\psi \mapsto \psi|_{\partial M}^{\partial M}$ on an identity neighbourhood in $\Diff (\partial M)$. The section $k$ is normalised as $k(\id_{\partial M}) =\id_M$.
\end{proposition}

\begin{proof}
We have already seen that $(\theta\circ G)^\vee$ is smooth. As $\Diff(M)$ is regular, $\Evol$ is smooth and thus $k$ is smooth as a composition of smooth maps and its domain~$\Omega$ is an identity neighbourhood. Unfolding the definition of $k$ in \eqref{splitting:section} and \eqref{curve:bdr:primitive}, we obtain directly from Lemma \ref{la:diff_ext} that $R \circ k = \id_\Omega$. The chart $\kappa$ used in the construction of the section $k$ send $\id_{\partial M}$ to the zero section, and the evolution sends the zero element to $\id_M$. Thus $k(\id_{\partial M}) = \id_M$.
\end{proof}

Note that the reason we only obtain a local section in Proposition \ref{prop:ex:smooth-sect} is that we needed to select an identity neighbourhood in which we could smoothly associate curves starting at the identity and ending in elements of the neighbourhood.

\begin{corollary} \label{cor:Product}
The short exact sequence \eqref{eq:sequence} of infinite-dimensional Lie groups admits sections. If $k \colon \Omega \rightarrow \Diff_0 (M)$ is a smooth section of $R$ with $k(\id_{\partial M}) = \id_M$, then for every open identity neighbourhood $U \subseteq \Diffpid (M)$, the set 
$$  k(\Omega) \cdot U \coloneq \{g= k(\omega) \cdot u \mid u \in U, \omega \in \Omega\}$$ is an open identity neighbourhood in $\Diff(M)$.
\end{corollary}

\begin{proof}
From \Cref{prop:ex:smooth-sect} it follows that the short exact sequence admits local sections in the sense that we can always find an identity neighbourhood on which a section of $R$ is smooth. The statement for the identity neighbourhood follows now by standard calculations for group extensions. The reader may for example consult \cite[Lemma 5.5 (c)]{bauer2023regularity}.
\end{proof}

In general we can not hope that the existence of local sections for the sequence \eqref{eq:sequence} can be improved to a splitting of the sequence. First of all, our results do not produce a global section (by construction of $\Omega$ in Section~\ref{setup:localisation} which depends on a suitable chart and is thus not easily globalised). Secondly, the extension operator $\theta$ in Lemma \ref{lem:ext_operator} multiplies by a cut-off function which leads to $\theta$ not being a Lie algebra morphism. We quantify the failure of $k$ to preserve the multiplication in the next section.

\subsection{The section as a (local) group morphism}
Note that the construction of the extension operator $\theta$ in Lemma \ref{lem:ext_operator} respects composition of diffeomorphisms in some neighbourhood of the identity. Recall the notation introduced in Section \ref{subsect_collar_thick}, where $N_s = h(\partial M \times [0,s])$ is a submanifold with smooth boundary of $M$ such that $\partial N_s = \partial M \sqcup h(\partial M \times \{s\})$. Let us write $h_s \coloneq h|_{\partial M \times [0,s]}^{N_s}$ for the diffeomorphism induced by the collar on $N_s$. We obtain the following multiplication preserving extension to diffeomorphisms of $N_s$: 

\begin{lemma}\label{lem:multiplicative}
Let $\Omega$ be the identity neighbourhood from Section~\ref{setup:localisation}, then 
$$\Theta_s \colon \Omega \rightarrow \Diff (N_s),\quad \psi \mapsto \Evol (Th_s \circ (\delta^r (f_\psi),\mathbf{0}) \circ h_s^{-1})(1),$$
where $\Evol$ is the evolution in $\Diff (N_s)$ makes sense, is smooth and satisfies:
\begin{enumerate}[\rm (1)]
\item $\Theta_s (\psi)|_{\partial M} = \psi$
\item $\Theta_s$ is a morphism of local Lie groups, i.e. $\Theta_s (\phi)\circ \Theta_s (\psi) = \Theta_s (\phi \circ \psi)$ and $\Theta_s(\varphi^{-1}) = \Theta_s (\varphi)^{-1}$, whenever both sides of the identities are defined. 
\end{enumerate}
\end{lemma}

\begin{proof}
 Since $\delta^r (f_\psi) (t) \in \VF(\partial M)$ and $h_s$ identifies $N_s$ with $\partial M \times [0,s]$, the vector field $Th_s \circ (\delta^r (f_\psi),\mathbf{0}) \circ h_s^{-1}$ on $N_s$ is tangent to the boundary of $N_s$, i.e.\ it is an element in $\VF(N_s;\partial N_s)$. This shows that $\Theta_s$ makes sense and arguments as in the proof of Lemma \ref{lem:ext_operator} and in Section \ref{setup:localisation} show that $\Theta$ is smooth (the argument being even simpler as we do not need the multiplication with the cutoff and subsequent extension). Property (1) follows now exactly as in the proof of Proposition \ref{prop:ex:smooth-sect}. For property (2) let us consider the Lie group isomorphism 
 $$E_s \colon \Diff (\partial M \times [0,s]) \rightarrow \Diff (N_s),\quad \phi \mapsto h_s \circ \phi \circ h^{-1}_s.$$ 
 Applying the Lie functor we obtain the Lie algebra isomorphism 
 $$e_s \coloneq \Lf (E_s) \colon \VF(\partial M) \rightarrow \VF(N_s;\partial N_s),\quad X \mapsto Th_s \circ X \circ h_s^{-1}.$$
 Now let us denote the time $1$ evolution of the Lie group $\Diff (\partial M \times [0,s])$ by $\evol$, then the naturality of the evolution (see \cite[Exercise 3.3.5]{schmeding_2022}) yields for the pushforward $(e_s)_\ast$ of $e_s$ on the space of smooth curves:
 \begin{equation}\label{eq:theta_rewrite}
 \Theta_s = \Evol \circ (e_s)_\ast \circ (\delta^r , \mathbf{0})(1) = E_s \circ \evol (\delta^r,\mathbf{0}).\end{equation}
 Observe that $\evol$ is associating to every (time dependent) vector field its time $1$ flow. This implies for the product vector fields of type $X \times \mathbf{0}$ that 
 \begin{equation}\label{eq:evol_rewrite}
 \evol (\delta^r (f_\psi),0)=(\Evol (\delta^r (f_\psi)(1),\id_{[0,s]})=(\psi,\id_{[0,s]})
 \text{ for all } \psi \in \Omega,\end{equation}
 where the last identity follows from the choice of $f_\psi$. Now since $E_s$ is a Lie group morphism the identities \eqref{eq:evol_rewrite} and \eqref{eq:theta_rewrite} combined yield (2). 
\end{proof}

In Lemma \ref{lem:multiplicative} we exploited that we did not need to multiply by a cutoff function. However, we can generalise the argument to quantify how much the section $k$ constructed in Proposition~\ref{prop:ex:smooth-sect} deviates from being a local group morphism. The key is again to consider a smooth cutoff function which is constant equal to $1$ in a collar neighbourhood of the boundary. As preparation, denote by $\Delta \Omega \coloneq \{(\phi,\psi) \in \Omega \times \Omega \mid \phi \circ \psi \in \Omega\}$. Note that since $\Omega$ is an open identity neighbourhood, $\Delta \Omega$ contains an open neighbourhood of $(\id_{\partial M},\id_{\partial M})$.
For an element $\phi \in \Diff (M)$, we let $\supp \phi$ be the closure of all $m \in M$ such that $\phi(m) \neq m$.

\begin{lemma}
For the short exact sequence \eqref{eq:sequence} we can choose a normalised smooth section $k \colon \Omega \rightarrow \Diff (M)$ of the restriction map $R$ such that 
\begin{align}\label{eq:almost_gp_mor}
 k(\phi \circ \psi) = k(\phi)\circ k(\psi) \circ \varepsilon (\phi,\psi) ,\quad \forall (\phi,\psi) \in \Delta \Omega
\end{align}
where $\varepsilon_k (\phi,\psi) \in \Diff (M)$ with $\supp \varepsilon_k (\phi,\psi) \subseteq M \setminus h(\partial M \times [0,s) )$. Finally, $\varepsilon_k$ is a smooth map in some neighbourhood of $(\id_M,\id_M) \in \Diff (M)^2$.
\end{lemma}

\begin{proof}
We shall prove the statement for the section $k$ which was constructed in Proposition \ref{prop:ex:smooth-sect}.
Let us first define $\varepsilon \colon \Delta \Omega \rightarrow \Diff (M)$ via the formula
\begin{align}\label{eq:errorterm}
\varepsilon (\phi,\psi) \coloneq k(\psi)^{-1} \circ k(\phi)^{-1}\circ k(\phi\circ \psi)
\end{align}
Then $\varepsilon_k$ makes sense and trivially satisfies \eqref{eq:almost_gp_mor}. By smoothness of the Lie group operations we can choose an identity neighbourhood $O \subseteq \Omega \subseteq \Diff (\partial M)$, such that $O \times O \subseteq \Delta \Omega$. On this neighbourhood \eqref{eq:errorterm} clearly defines a smooth map as $k$ is smooth on its domain and composition and inversion is smooth in the respective diffeomorphism groups.

It remains to prove that for any pair $(\phi,\psi) \in \Delta \Omega$ the support of $\varepsilon (\phi,\psi)$ is contained in $M\setminus h(\partial M \times [0,s))$. For this it suffices to note that for $x \in N_s$ the cutoff function from \eqref{ChiFunction} satifies $\chi (x) = 1$. So for a vector field $F \in \VF (\partial M)$ we have with the notation of Lemma \ref{lem:multiplicative} the identity 
$$\theta (F)(x)=\hat{F}(x) = Th_s \circ (F,\mathbf{0})\circ h_s^{-1}(x), \quad \forall x \in N_s.$$
Thus for $x \in N_s$ we obtain $k(\varphi)(x)=\Theta_s (\varphi) (x)$. By Lemma \ref{lem:multiplicative} $\Theta_s$ is a morphism of local Lie groups, hence for $(\varphi ,\psi) \in \Delta \Omega$ and $x \in N_s$ we obtain
\begin{align*}
\varepsilon (\varphi, \psi) (x) &= k(\psi)^{-1} \circ k(\phi)^{-1}\circ k(\phi\circ \psi) (x)\\ &= k(\psi)^{-1} ( k(\phi)^{-1}(\Theta_s(\phi\circ \psi)(x)) \\ &= \Theta_s(\psi)^{-1}(\Theta_s(\phi)^{-1}(\Theta_s(\phi)\Theta_s (\psi)(x))=x
\end{align*}
So only $x$ in the boundary of $N_s$ can be contained in the support of $\varepsilon (\varphi,\psi)$. 
\end{proof}

\begin{remark}
Let $\tilde{R}\colon \Diff (M) \rightarrow \Diff(\partial M)$ be the restriction map. As a consequence of Lemma \ref{eq:errorterm}, one can show that for every $\tilde{R}(\Phi) = \phi$ there exists a local section $k_\phi \colon \phi \Omega \rightarrow \Diff (M), k_\phi (\psi) \coloneq \Phi \circ k(\phi^{-1}\circ \psi)$. Hence the splitting of \eqref{eq:sequence} generalises to a splitting where the right hand side of the sequence is replaced by the image of the restriction map. It is however not clear to the authors how to describe the image of $\tilde{R}$ and whether the restriction is surjective onto $\Diff (\partial M)$.
\end{remark}

\section{Proof of the controllability results}\label{sect:controllability}
\subsection{Why the results are not geometric} \label{sec:NotGeometric}
We will first explain how Theorem~\ref{th:control} does not depend on the choice of geometry. We can therefore modify the Riemannian metric $g$ without changing the controllability properties.

\begin{proposition}
Let $(M,g)$ be a Riemannian manifold with boundary, outward normal vector field $\nu$ and Levi-Civita connection $\nabla$. Define $\tilde g$ as another Riemannian metric on $M$ and let $\tilde \nu$ be its outward normal vector field. Let $\tilde \nabla$ be any other affine connection. Finally, let $Z$ be a vector field vanishing on the boundary and $x \in \partial M$ be an arbitrary point. Then
$$\langle \nabla_{\nu} Z, \nu \rangle_g(x) \neq 0,$$
if and only if
$$\langle \tilde \nabla_{\tilde \nu} Z, \tilde \nu \rangle_{\tilde g}(x) \neq 0,$$
\end{proposition}
\begin{proof}
We will show that in fact $\nabla_{\nu} Z(x)$ is parallel to $ \tilde \nabla_{\tilde \nu} Z(x)$. The proof will then follow from the fact that the inner product is non-zero relative to the outwards normal vector if and only if the other vector is not tangent to the boundary, regardless of the choice of metric.

Let us choose a coordinate system $(x_1, \dots, x_n)$ around the point $x$, identifying it with $0$, and defining it such that the boundary is $x_n = 0$ and the interior is $x_n >0$. Let us write $Z = \sum_{j=1}^n Z^j \partial_{j}$, $\nu = \sum_{j=1}^n \nu^j \partial_j$ and $\tilde \nu = \sum_{j=1}^n \tilde \nu^j \partial_j$. We remark that we must have $\nu^n <0$ and $\tilde \nu^n <0$ since they are outward normal vector fields. If $\Gamma_{ij}^k$ are the Christoffel-symbols of $\nabla$ in these coordinates, then
$$\nabla_{\nu} Z(0) = \sum_{j,k=1}^n \nu^j(0) \partial_j Z^k(0) \partial_k + \sum_{i,j,k=1}^n \nu^i(0) Z^j(0) \Gamma_{ij}^k \partial_k = \nu_n (0) \sum_{k=1}^n \partial_n Z^k(0) \partial_{k}.$$
Here we have used that since $Z$ vanish at $0$, then the Christoffel symbols, and hence choice of connection plays no role. Furthermore, all derivatives in the direction of the boundary will vanish since $Z$ is constantly zero there. Similarly
$$\tilde \nabla_{\tilde \nu} Z(0) = \tilde \nu_n (0) \sum_{k=1}^n \partial_n Z^k(0) \partial_{k}.$$
It finally follows that the vector fields are parallel, giving us the result.
\end{proof}

\subsection{Localisation lemma and proof of controllability} Let $M$ be a compact manifold with smooth boundary $\partial M$. If $U$ is an open set in $M$, then $C^\infty_{\str}(U,M)$ denotes smooth functions such that $f(U \cap \partial M) = f(U) \cap \partial M$. Our results then follow as a consequence of the following lemma.
\begin{lemma}[Localisation lemma] \label{lemma:nbh}
Let $\scrF$ be a collection of vector fields satisfying the assumptions of Theorem~\ref{th:control}. Then for any point $x \in M$, we have a neighbourhood $U$ of $x$ and an open set $\mathscr{U}$ in $C^\infty_{\str}(U, M)$ of the identity generated by~$\hat{\mathscr{G}}|_{U}$.
\end{lemma}
Before considering the proof of Lemma~\ref{lemma:nbh}, we will show how it implies our main results on controllability, using an approach mirroring \cite{Agrachev}.

\begin{proof}[Proof of Theorem~\ref{th:control}]
For any $x \in M$, define $U_x$ as the corresponding neighbourhood with the desired properties. Consider an open cover $\bigcup_{x} U_x$, define a finite subcover $U_1 \cup \cdots \cup U_l$ and introduce a partition of unity $\lambda_1,\dots, \lambda_l$ subordinate to this subcover. For $j =1,2 \dots, l$, let $\mathscr{U}_j$ be the corresponding open neighbourhoods, such that $\mathscr{U}_j$ is an open set in $C^\infty(U_x, M)$.

Now let $\mathscr{O}$ be an open, connected neighbourhood in $\Diff(M)$ of the identity. Define $\mathscr{O}_j$ as the subset of $\mathscr{O}$ with support in $U_j$. It then follows that $\scrO \subseteq \scrO_1 \circ \scrO_2 \circ \cdots \circ \scrO_l$. Namely, for any diffeomorphism  $\phi \in \scrO$, let $s \mapsto \varphi_s$ be a curve in $\scrO$ with $\varphi_0 = \id_M$ and $\varphi_1 = \phi$. Define now $\psi_j$ and $\phi_j$, $j=1,2, \dots, l$, by
$$\psi_j(x) = \varphi_{\lambda_1(x) + \cdots + \lambda_j(x)}(x),$$
$$\phi_1 = \psi_1, \qquad \phi_{j+1} = \psi_{j+1} \circ \psi_j^{-1}.$$
Then $\phi_{j} \in \scrO_j$ and $\phi = \phi_l \circ \cdots \circ \phi_1$. We can now define a new neighbourhood $\hat{\scrO}$ generated by $\scrO_j \cap \scrU_j$. Since $\hat{\scrO}$ is contained in $\hat{\scrG}$, the result follows.
\end{proof}

For the remainder of the section, we want to prove the result Lemma~\ref{lemma:nbh} for all different types of points in $M$. For all of these cases, we will need the following result from \cite[Prop~4.1]{Agrachev}.

\begin{lemma} \label{lem:Xn}
If $X_1,\dots, X_n$ are vector fields on $\mathbb{R}^n$ such that
$$\spn \{ X_1(0), \dots, X_n(0) \} = T_0 \mathbb{R}^n.$$
Then there exists a relatively compact neighbourhood $V$ of $0$ and a neighbourhood $\mathscr{V}$ of $C^\infty(V,\mathbb{R}^n;0)$, such that any $F \in \mathscr{V}$ can be written as
$$\psi = e^{f_1 X_1} \circ e^{f_2 X_2} \circ \cdots \circ e^{f_n X_n}|_V, \qquad f_j \in C^\infty(\mathbb{R}^n), \qquad f_j(0) = 0.$$
\end{lemma}
Here and in the following we use $$C^\infty(V, \mathbb{R}^n;0) \coloneq \{f \in C^\infty (V,\mathbb{R}^n) \mid f(0)=0\}$$ to denote the space of smooth functions mapping zero to itself. Note that it is 
 closed linear subspace of $C^\infty (V,\mathbb{R}^n)$. We will also use Sussmann`s Orbit Theorem \cite{sussmann1973orbits}, see also \cite[Chapter~5]{AgSa04} and \cite[Theorem 1]{McK07}, stating that if $\Orb_x = \scrG(x)$ is the orbit of $\scrG$ through $x$, then
$$T_y \Orb_x = \{ \phi_* Y(y) \, : \, \phi \in \scrG, Y \in \scrF \}.$$

\subsection{A lemma on flows of vector fields}
We will use the following simple observation repeatedly below. Recall that $H^n = \{ x_n \geq 0\} \subseteq \R^n$.
\begin{lemma} \label{lemma:varphi}
Let $\hat Z$ be a vector field on $\R^n$ such that $dx_n(\hat Z)(0) \neq 0$. Then there exists a function $\varphi: \R^n \to \R$ of compact support such that
\begin{enumerate}[\rm (a)]
\item $\varphi|_V =1$ for some neighbourhood $V$ of $0$.
\item If $\hat g \in C^\infty(\mathbb{R}^n)$ is a function such that $e^{\hat g\varphi \hat Z}$ preserves $\partial H^n$, then on the support of $\varphi$ we can write $\hat g(x) = x_n g(x)$ for some $g \in C(\R^n)$.
\end{enumerate}
\end{lemma}
In order to explain why we need this lemma, we note that for the case of $n=2$, we have
$$e^{- 2\pi x_2 \partial_1 + 2\pi x_1 \partial_2} = \id,$$
in spite of the fact that the vector field does not vanish on $\partial H^n$ outside of zero. For later purposes, we will need to exclude cases where the points can flow away from the boundary only to return later. So we use a cutoff function to prevent fast escape from a region with inward pointing derivative of the vector field $\hat Z$.

\begin{proof}
We will assume that $dx_n(\hat Z)(0) >0$, as the case where we have the opposite sign can be dealt with similarly. Define first
$$V_0 = \{ x  \in \R^n \, : \, dx_n(\hat Z)(x) >0 \},$$
and let $V'$ be a connected, relatively compact neighbourhood of $0$ in $\partial H^n \cong \R^{n-1}$ such that $\overline{V}'$ is contained in $V_0$. By compactness of $\overline{V}'$ we may pick $\varepsilon >0$ such that for all $x'  \in \overline{V}'$, and  $t \in (-\ve, \ve)$ we have 
$$e^{t\hat{Z}}(x') \in V_0 \qquad \text{for any $t \in (-\ve, \ve)$.}$$
Define for this choice of $\varepsilon$ the set
$$V_1 = \{ e^{t\hat Z}(x') \, : \, x' \in V', t \in(-\ve, \ve) \}.$$
Since the flows of $t\hat{Z}$ are transversal to the boundary for all points in $\partial H^n \cap V_0$, the set $V_1$ is an open neighbourhood of $V'\times \{0\}$.
Hence we can pick $\varphi \in C^\infty (\R^{n})$ as a function of compact support which satisfies the following properties
\begin{enumerate}[$\bullet$]
\item there is a 
$0$-neighbourhood $U$ such that $\varphi|_U \equiv 1$ and $\overline{U} \subseteq V_{1}$.
\item $\varphi >0$ on $V_{1}$,
\item $\varphi \equiv 0$ outside of $V_{1}$ and on $\partial V_1$.
\end{enumerate}
We note that the flow $\Phi_t \coloneq e^{t\varphi Z}$ satisfies
$$\Phi_t (V_1)  \subseteq V_1, \qquad \Phi_t|_{\R^{n}\setminus V_1} = \id_{\R^{n} \setminus V_1}.$$
Let us write now $(y',y_n) = e^{r \hat Z}(x')$ for $y' \in \R^{n-1}$, $y_n \in \R$ and $x' \in V', r \in (-\varepsilon, \varepsilon)$. Since $\hat Z$ is positive on $V_1$ in the $n$th coordinate direction, we deduce $\text{pr}_n(e^{t\varphi \hat Z}((y',y_n))) > y_n$ whenever $t >0$ and $(y',y_n) \in V_1$. Similarly we have a decrease in the $n$-th coordinate when $t$ is negative. It follows that if $e^{\hat g\varphi \hat Z}$ preserves $\partial H^n$, $\hat g$ has to vanish on $\partial H^n \cap V_1$, which gives the result.
\end{proof}

\subsection{Localisation lemma for an interior point} \label{sec:InteriorPoint}
For an interior point, we can use the results in \cite{Agrachev}. Using a local chart, we may assume that $x =0$ in $\mathbb{R}^n$. By the Orbit Theorem, we can find $\phi_j \in \scrG$, $Y_j \in \scrF$ such that $X_j = \phi_{j,*} Y_j$, $j=1,\dots,k$ spans $T_0 \mathbb{R}^n$ at $0$. Using Lemma~\ref{lem:Xn}, this means that we can find a neighbourhood $\mathscr{U}_{00}$ in $C^\infty(U,M;0)$ around the identity such that
\begin{align*}
\phi & = e^{f_1 X_1} \circ \cdots \circ e^{f_{n} X_{n}}  \\
& = \phi_1 \circ e^{(f_1 \circ \phi_1) Y_1} \phi_1^{-1} \circ \cdots \circ \phi_{n} \circ e^{(f_{n} \circ \phi_{n}) Y_{n}} \circ  \phi_{n}^{-1} |_U \in \hat{\scrF}|_U.
\end{align*}
Next, let $Q(s_1, \dots, s_n) = e^{s_1 X_1} \circ \cdots \circ e^{s_n X_n}$, and let $V$ be a neighbourhood of $s=0$ such that $s\mapsto Q(s_1, \dots, s_n)(0)$ is a diffeomorphism onto $Q(V)(0)$ with inverse $q$. Now define a neighbourhood $\mathscr{U}_0$ of the identity in $C^\infty(U,M)$ consisting of diffeomorphism satisfying
$$\phi(0) = y, \qquad Q(q(y))^{-1} \circ \phi \in \mathscr{U}_{00}.$$
The result follows

\subsection{Localisation lemma for a boundary point} \label{LLBoundary} 
We will now show how Section~\ref{sec:InteriorPoint} can be modified to work on a boundary point $x \in \partial M$. It is actually sufficient to prove it at one point for each boundary component. To give some more details, let $x \in \partial M$ be a boundary point with neighbourhood $U$ and neighbourhood of the identity $\scrU$ in $C^\infty_{str}(U,M)$ contained in $\hat \scrG|_{U}$. Let $y \in \partial M$ be a boundary point in the same connected component and let $\psi$ be a map in $\scrG$ with $\psi(x) = y$, which exists by (I). Define $U_y = \psi(U)$ and note that $\psi \circ \scrG|_{U} \circ \psi^{-1} = \scrG|_{U_y}$ and so $\scrU_y =\psi \circ \scrU \circ \psi^{-1}$ is a neighbourhood of $\id_{U_y}$ contained in $\scrG|_{U_y}$. 

As a consequence, we only need to prove the statement for a boundary point~$x$ where we have a vector field $Z$ vanishing on the boundary close to $x$ such that $\langle \nabla_\nu Z, \nu \rangle_g(x) \neq 0$. By using a local chart in a neighbourhood $U_0$ of $x$, we can reduce our considerations to the case when~$U_0$ is a neighbourhood of~$x=0$ in~$H^n = \{ x_n \geq 0\} \subseteq \mathbb{R}^n$, and with $Z|_{U_0 \cap \partial H^n} =0$. By Section~\ref{sec:NotGeometric}, we can also change to the geometry of $H^n$, meaning that $Z$ satisfies $dx^n(\partial_n Z) \neq 0$.

Let $U^+_0$ be a connected neighbourhood of $0$ in $\R^n$ which has $U^+_0 \cap H^n = U_0$. Using a bump function with support on $U^+_0$ which is constant 1 close to $0 \in \R^n$ and Seeley's extension argument in \cite{Seeley64} (see \cite[Theorem 3.1]{Han21} for a more general version working directly on open subsets of half-spaces), we can consider $C^\infty(U_0,H^n)$ as a subset of $C^\infty(U^+_0, \R^n)$. By the same argument, we can extend our vector fields $\scrF|_{U_0}$ to $U^+_0$ which we will denote by $\scrF^+$.

Let $\Orb_0$ be the orbit of $\scrF^+$ through $0$. By (I) and the Orbit Theorem
$$T_0 \Orb_0 = T_0 (\partial H^n),$$
so we can find elements $\phi_j \in \Gr(\scrF^+)$, $Y_j \in \scrF^+$, $j = 1, \dots, n-1$, such that $X_j = \phi_{j,*} Y_j$ satisfies
$$\spn\{ X_1(0), \dots, X_{n-1}(0) \} = T_0 (\partial H^n).$$
We remark that since the flows of vector fields in $\scrF^+$ never pass through the hyperplane $x_n = 0$, each restriction of an element in $\Gr(\scrF^+)$ to $x_n \geq 0$ will be in $\Gr(\scrF|_{U_0})$.
Finally, since $Z$ vanishes on the boundary, we can write $Z=x_n\hat{Z}$ for some vector field $\hat{Z}$ satisfying $dx^n(\hat Z)(0)$ and hence being non-vanishing close to zero. We can thus consider the vector field $\varphi \hat Z$, where $\varphi$ is as in Lemma~\ref{lemma:varphi}, and that satisfies
$$\spn\{ X_1(0), \dots, X_{n-1}(0), \varphi(0) \hat Z(0) \} = T_0 \R^n.$$
By Lemma~\ref{lem:Xn}, there is a neighbourhood $U_1^+$ of $0$ in $\R^n$ and a neighbourhood $\scrU_{1}$ in $C^\infty(U_1^+,\R^n;0)$ such that any element in $\phi \in \scrU_1$ can be written as 
\begin{align*}
\phi & = e^{f_1 X_1} \circ \cdots \circ e^{f_{n-1} X_{n-1}} \circ e^{\hat g \varphi \hat Z} |_{U_1^+} \\
& = \phi_1 \circ e^{(f_1 \circ \phi_1) Y_1} \circ \phi_1^{-1} \circ \cdots \circ \phi_{n-1} \circ e^{(f_{n-1} \circ \phi_{n-1}) Y_{n-1}} \circ  \phi_{n-1}^{-1} \circ e^{\hat g \varphi \hat Z} |_{U_1^+}
\end{align*}
If $U_1 = U_1^+ \cap H^n$ and $\phi \in C^\infty_{str}(U_1,H^n;0) \cap \scrU_1$, then since $\phi$ preserves the boundary we must have that $e^{\hat g \varphi \hat Z}$ preserves the boundary. We deduce that $\hat g \varphi \hat Z =x_n g \varphi \hat Z = g \varphi Z$, giving us
\begin{align*}
\phi |_{U_1}
& = \phi_1 \circ e^{(f_1 \circ \phi_1) Y_1} \circ \phi_j^{-1} \circ \cdots  \circ \phi_{n-1} \circ e^{(f_{n-1} \circ \phi_{n-1}) Y_{n-1}} \circ  \phi_{n-1}^{-1} \circ e^{g \varphi Z}|_{U_1} \in \hat{\mathscr{G}}|_{U_1}
\end{align*}
We next observe that if
$$Q(s_1, \dots, s_{n-1} ) = e^{s_1 b X_1} \circ \cdots \circ e^{s_{n-1} b X_{n-1}},$$
then $s \mapsto Q(s_1, \dots, s_{n-1})(0)$
is a local diffeomorphism into $\partial H^n$. Let $V$ be a neighbourhood of $0$ such that $V \mapsto Q(V)(0)$, $s \mapsto Q(s)(0)$ is a diffeomorphism. Write its inverse as $q: Q(V)(0) \to V$. Define a neighbourhood $\mathscr{U}_0$ in $\Diff(U)$, consisting of $\phi$ such that
$$\phi(0) =y  \in Q(V)(0), 
\qquad Q(q(y))^{-1} \circ \phi \in \mathscr{U}_{1}.$$
This completes the proof.

\subsection{Proof of Corollary \ref{cor:exp_gen}}
The first statement is a result of Theorem~\ref{th:control}. For the second statement, we start with $\scrF = \hat{\scrF} = \VF(M)$. We complete the proof as above, but for the neighbourhood of the boundary points as in Section~\ref{LLBoundary}, after localising to a neighbourhood of $H^n$ through a choice of local coordinate system $U$, we may choose $X_1 = \partial_{x_1}$, $\dots$, $X_{n-1} = \partial_{x_{n-1}}$ and $Z = x_n \partial_{x_n} =x_n \hat Z$. We know that for elements in the support of $U$, we have
$$\psi = e^{\hat f_1 \partial_{x_1}} \circ \cdots \circ e^{\hat f_{n-1} \partial_{x_{n-1}}} \circ e^{\hat g\partial_{x_{n-1}}}|_U$$
If $\psi$ equals $\id$ on the boundary close to zero, we need that $\hat f_j = x_n f_{j+1}$, and $\hat g = x_n g$, giving us
$$\psi = e^{x_n f_1 \partial_{x_1}} \circ \cdots \circ e^{x_n f_{n-1} \partial_{x_{n-1}}} \circ e^{ g \partial_{x_n}} \in \Gr(\VF^{\partial,\id}(U)),$$
which gives the result.

\section{Examples} \label{sec:Examples}
\subsection{A counter-example}\label{sect:counterex}
We will give a counter example showing that it is not sufficient to just require that $\scrG$ acts transitively on the interior and on connected components of the boundary.
Our example will be $M = [0,1]$ with coordinate~$x$, but taking products with other manifolds, we can obtain a counter-example in any dimension. Let $\scrF = \spn \{ Y\}$, where $Y$ is the vector field vanishing on the boundary satisfying
$$Y(x) = \sinh(x(x-1)) \partial_x =: b(x) \partial_x,$$
Observe that $\scrG$ then acts transitively on the interior, while preserving the two boundary points. Consider any diffeomorphism
$\varphi_t = e^{tf Y} =\varphi_t = e^{tfb \partial_x}$.
We then observe that
$$\left.\frac{d^2}{dx dt} \varphi_t \right|_{x=0} =0,$$ meaning that $\frac{d}{dx} \varphi_t(0) = 1$ for all such functions. On the other hand, if we consider $X = \sin(\pi x)\partial_x$ then
\begin{equation} \label{FFlow} \psi_t(x) = e^{tX}(x) = \frac{2}{\pi} \tan^{-1}\left(e^{\pi t} \tan\left(\frac{\pi x}{2} \right)\right) \end{equation}
has derivative
$$\frac{d}{dx} \psi_t(x) = \frac{e^{\pi t}}{1 + (e^{2\pi t}-1) \sin^2 \frac{\pi x}{2}},$$
and in particular $\frac{d}{dt} \psi_t(0) = e^{\pi t} >1$ for $t >0$. It follows that $\psi_t$ for $t >0$, can not be written as $e^{fY}$ for any function $f$ on $M$.

\subsection{Sub-Riemannian manifolds with boundary}
For a general manifold $N$ without boundary, we say that a subbundle $E \subseteq TN$ is \emph{bracket-generating} if vector fields with values in $E$ along with their iterated Lie brackets span $TN$. We use the same terminology if $E$ is a \emph{rank-varying subbundle} in the sense of \cite[Definition~1]{agrachev2010two}, where $E$ is allowed to have different rank depending on the point.

Let $M$ be a compact submanifold with smooth boundary. Let $E \subseteq TN$ be a subbundle of the tangent bunlde, and define a collection of vector fields
$$\scrF = \hat \scrF =\{ X \in \Gamma(E) \, : \, X(x) \in T_x \partial M \text{ for any $x \in M$}\}.$$
We then have the following result.
\begin{proposition}
Write $\rank E = k$ and define $E' = E\cap T\partial M$. Assume that
\begin{enumerate}[\rm (i)]
\item $E|_{\intOP M}$ is bracket-generating on $\intOP M$.
\item $E'$ is bracket-generating on $\partial M$.
\item For each boundary component, there is one point $x \in \partial M$ such that $E'_x$ has rank $k-1$.
\end{enumerate}
Then $\scrG = \Gr(\scrF) = \Gr(\hat \scrF)$ coincides with $\Diff(M)$.
\end{proposition}
We remark that $E'$ will be a rank-varying subbundle in general.
\begin{proof}
For convenience, we introduce an auxillary Riemannian metric $g$ on $M$, and let $\nu$ be the corresponding outward normal vector. By the Chow-Rashevskii theorem \cite{Cho39,Ras38}, which is a special case of the Orbit theorem, we have that $\scrG$ acts transitive on $\intOP M$ and on $\partial M$, meaning that condition (I) is satisfied. Next, around any boundary point $x \in \partial M$ such that $\rank E_x = k-1$, we choose a neighbourhood sufficiently small so that there exists an orthonormal basis $X_1, \dots, X_{k-1}, \hat Z$ of $E$ such that $X_1, \dots, X_{k-1}$ spans $E'$ at the boundary. By possibly shrinking $U$, we can assume that there is a function $\phi$ such that $\phi|_{U \cap \partial M} =0$, $X_j\phi|_{U \cap \partial M} =0$ and $\nu \phi|_{U \cap \partial M} >0$. We can then define $Z = \phi \hat Z$ which satisfies (II) at $x$.
\end{proof}

\begin{example}{(Euclidean ball in the Heisenberg group)}
Consider $\mathbb{R}^{2n+1}$ with the standard Euclidean norm and write the coordinates as $(x,y,z) \in \mathbb{R}^{2n+1}$, $x,y \in \mathbb{R}^n$, $z \in \mathbb{R}$. We consider the closed Euclidean unit ball $M = \{ \|(x,y,z)\| = 1 \, ; \, (x,y,z) \in \mathbb{R}^{2n+1}\}$. Introduce vector fields on $\mathbb{R}^{2n+1}$ by
$$X_j = \partial_{x_j} - \frac{y_j}{2} \partial_{z}, \qquad Y_j = \partial_{y_j} + \frac{x_j}{2} \partial_{z}, \qquad j=1, \dots, n.$$
For smooth functions $a,b:\mathbb{R}^{2n+1} \to \mathbb{R}^n$, we define
$$aX + b Y = \sum_{j=1}^n (a_j X_j + b_j Y_j).$$
Then
\begin{align*}[a X+ b Y , \tilde a X + \tilde b Y] & =((a X + b Y) \tilde b - (\tilde a X + \tilde bY ) b) Y \\
& \qquad + ((a X + b Y) \tilde a - (\tilde a X + \tilde bY ) a) X \\
& \qquad + (\langle a, \tilde b \rangle - \langle b, \tilde a \rangle ) \partial_z.\end{align*}
Hence $E = \spn\{ X_1, \dots, X_n, Y_1, \dots, Y_n \}$ is bracket-generating everywhere, including in $\intOP M$. For $E' = E \cap T\partial M$, write $X_j' = X_j|_{\partial M}$ and $Y_j' = Y_j|_{\partial M}$. Then for $(x,y,z) \in \partial M$,
\begin{align*} E'_{x,y,z} &= \spn \left\{a X' + b Y' \, : \, \langle a, x \rangle + \langle b, y \rangle + \frac{z}{2} (-\langle y, a \rangle + \langle x,b \rangle ) =0 \right\} \\
& = \spn \left\{ a X' + b Y' \, : \, \langle a , e_1(x,y,z) \rangle + \langle b , e_2(x,y,z) \rangle =0 \right\}\end{align*}
where $|x|^2 + |y|^2 \leq 1$, $z = \pm \sqrt{1-|x|^2 - |y|^2}$ and
$$e_1(x,y,z) = x- \frac{zy}{2} , \qquad e_2(x,y,z) = y + \frac{zx}{2} .$$
Note that $e_1$ and $e_2$ vanish simultaneously if and only if $x= y =0$. Hence $\rank E' = 2n-1$ outside of the point $(0,0, \pm 1)$. Condition (iii) is hence satisfied. At the north and south pole, we have $E_{0,0,\pm 1} = T_{0,0,\pm 1} \partial M$, but for orther proints, the missing direction in $E$ is spanned by
$$Z' = \partial_z - \frac{z}{|e_1|^2 + |e_2|^2} (e_1 X+ e_2 Y).$$
We need to show that we can generate this missing direction outside of the points $(0,0,\pm 1)$ in order to have the bracket-generating condition, in other words, we have to show that $E'' =E|_{\partial M \setminus (0,0,\pm 1)}$ is bracket-generating. We consider three cases below.
\begin{enumerate}[(a)]
\item If $n=1$, then $E''$ has rank 1 and is integrable, meaning that there is a one-dimensional foliation of $\partial M \setminus \{ 0,0,\pm 1)\}$ such that $\scrG$ will act transitively on its leafs. Hence, $\scrG$ will not act transitively on the boundary.
\item If $n \geq 3$, then we can find a unit vector $a: \mathbb{R}^{2n+1} \to \mathbb{R}^n$ such that it is orthogonal to $e_1$ and $e_2$. Since now $aX'$ and $aY'$ are sections of $E''$, and
$$[aX',aY'] = Z' \, \mod E'',$$
we have the bracket-generating condition. It follows that $\scrG$ acts transitively on the boundary.
\item If $n=2$, we will need to find $a,b:\mathbb{R}^{5} \to \mathbb{R}^2$ such that on $\partial M$,
$$\langle a, e_1 \rangle =0, \qquad \langle b, e_2 \rangle =0, \qquad \langle a,b \rangle =1.$$
which is equivalent to $aX'$ and $bY'$ being tangent tangent to $\partial M$ as well as satisfying $[aX,bY] = Z' \mod E'$. The only obstruction for being able to complete this choice would be if $e_1$ and $e_2$ where both non-zero and orthogonal at a point. However, if we choose $x$ and $y$ non-zero, orthogonal and of equal length, then $e_1(x,y)$ and $e_2(x,y)$ are non-zero and orthogonal. Hence, we cannot make $E''$ bracket-generating in this case.
\end{enumerate}
\end{example}

\printbibliography

\end{document}